\documentclass{amsart}

\usepackage{graphicx}
\usepackage{systeme}
\usepackage{setspace}
\newtheorem{theorem}{Theorem}[section]
\newtheorem{lemma}[theorem]{Lemma}

\theoremstyle{definition}

\newtheorem{corollary}{Corollary}[theorem]

\theoremstyle{remark}

\numberwithin{equation}{section}

\begin{document}

\title[About the Polignac's conjecture]{About the Polignac's conjecture}

\author{Douglas Azevedo}
\address{UTFPR-CP. Av. Alberto Carazzai 1640, centro, caixa postal 238,  86300000, Cornelio Procopio, PR, Brasil.}
\email{dgs.nvn@gmail.com}






\begin{abstract}
In this note we present a method to bound gaps between infinitely many consecutive prime numbers via the divergence of the series of reciprocals of the prime numbers and a  consequence of a test for convergence of series of positive numbers. 
Using this alternative method we indicate how our results are related to the  Polignac's conjecture .
\end{abstract}

\maketitle

\section{Introduction}

The study of the  gap between consecutive prime numbers, that is, the  investigation of the behavior of the sequence $g_n=p_{n+1}-p_n$, has attracted attention of the  most prominent mathematicians in the area. Here, 
 $p_n$ denotes the $n$th prime number.
Two of the most important problems related to this subject are the twin-prime conjecture and the Polignac's conjecture. The first one states that there exist infinitely many primes for which $g_n=2$ and the second conjecture generalizes the first one, by stating that for any positive even integer $k$, there  exist infinitely many primes for which $g_n=k$. As far as we know, until this moment  both conjectures were still not proved or disproved.  However, several works have contributed to the progress of this subject. Let us indicate some of them.

 In (\cite{gold}) Dan  Goldston, J\'anos Pintz, and Cem Yıldırım (also indicated as GPY)  presented a solution for 
  a long-standing open problem. They proved that there are infinitely many primes for which the gap to the next
prime is as small as we want compared to the average gap between consecutive primes. They showed that 
\begin{equation}\label{GPY}
\liminf_{n\to\infty} \frac{p_{n+1}-p_{n}}{\ln(p_n)}=0.
\end{equation}
There, the approach adopted is usually referred   as the
level of distribution of primes in arithmetic progressions and, with an additional assumption on this level of distribution the showed that
\begin{equation}\label{GPYa}
\liminf_{n\to\infty} p_{n+1}-p_{n}\leq 16.
\end{equation}
Latter, in \cite{gold1} the same authors considerably improved \eqref{GPY}
proving that 
\begin{equation}\label{GPY1}
\liminf_{n\to\infty} \frac{p_{n+1}-p_{n}}{\sqrt{\ln(p_n)}\ln(\ln(p_n))^2}<\infty.
\end{equation}
This result shows that there exist pairs of primes nearly
within the square root of the average spacing.

In 2013, Yitang Zhang \cite{zhang} published his celelbrated paper providing  the first proof of finite gaps between prime numbers. There it is shown that
\begin{equation}\label{zhang}
\liminf_{n\to\infty} p_{n+1}-p_{n}\leq 7.10^7.
\end{equation}
His approach was a refinement of the  work of Goldston, Pintz
and Yildirim on the small gaps between consecutive primes \cite{gold} and 
a major ingredient of the proof is a stronger version of the Bombieri-Vinogradov
theorem \cite{zhang}. A nice exposition of the Zhang's proof can be found in \cite{gran}. 

The improvement of the Zhang's numerical bound on the gaps  was obtained right after. For instance, the work of Polymath8 (\cite{polymath8}) and  Maynard (\cite{maynard}) 
presented  a reduction of Zhang's bound to $4680$ and $600$, respectively.
In particular, in Maynards work the  proofs involved quite different methods 
to Zhang and brought
the upper bound down to 600 using the Bombieri - Vinogradov Theorem (not
Zhang’s stronger alternative) and an improvement on GPY results. We refer to \cite{musson} and references there in for more information about the developments 
of the investigation about gaps of  prime numbers.

 The main feature of this note is a new approach to deal with  gaps of infinitely many consecutive prime numbers.  Our method
 involves the divergence of the series of the reciprocals of the prime numbers and a consequence of a version of a test for convergence of series of positive terms. 
 
  As it is presented at the final of this note,  our method may be used as an alternative  to  deal with  the Polignac's conjecture .

Let us start with the following classical result.  

\begin{lemma}\label{primes}
The series $\sum \frac{1}{p_n}$ diverges.
\end{lemma}

The next result is a consequence of the prime number theorem (see \cite[p. 79]{apostol}). 
We will make use of the asymptotic notation $u(x)\sim v(x)$, meaning that
$$\lim_{x\to\infty}\frac{u(x)}{v(x)}=1.$$

\begin{lemma}\label{PNT}
$$p_n\sim n\ln(n).$$
\end{lemma}

See, for instance, \cite[p.16]{apostol}, for a proof.

Next we will use $f(x)=o(g(x))$, as $x\to\infty$ to denote 
$$\lim_{x\to\infty}\frac{f(x)}{g(x)}=0.$$

Let us define a class of summable positive sequences that will play an importante role in this note. We will say that a pair
$(\{b_n\},r)$ satisfies the $*$-condition, in which $\{b_n\}$ is a summable sequence of positive terms and $r>0$, if there exists $n_0>0$ for which
$$\frac{b_n}{b_{n+1}}= 1+\frac{r}{n\ln(n)}+o\left(\frac{1}{n\ln(n)}\right),$$
for all $n>n_0$.

The next result is a test for convergence of series of positive terms.

\begin{lemma}\label{LEMA}
Let $(\{b_n\},r)$ be a pair satisfing the $*$-condition, for some $r>0$.
If $\{a_n\}$ is a sequence of positive terms  such that there exists a $n_1>0$ for which
$$\frac{a_n}{a_{n+1}}\geq 1+\frac{r}{n\ln(n)}+o\left(\frac{1}{n\ln(n)}\right),$$
for all $n>n_1$, then $\{a_n\}$ is summable.
\end{lemma}
\begin{proof}
Since $\sum_{n=1}^{\infty}b_n$ converges and 
$$\frac{b_{n}}{b_{n+1}}= 1+\frac{r}{n\ln(n)}+o\left(\frac{1}{n\ln(n)}\right)\leq \frac{a_n}{a_{n+1}},$$
for all $n>\max\{n_0,n_1\}$, the  comparsion test of second kind (see \cite[p. 114]{knopp}) implies  that
$\sum a_n$ converges.

\end{proof}

An immediate consequence of the previous lemma is the following key result for the main result of this note. The idea is to use the contrapositive argument in Lemma \ref{LEMA}.

\begin{corollary}\label{coro}
Let $\{a_n\}$ be a sequence of positive terms. If $\sum a_n$ diverges
then for every pair $(\{b_n\},r)$ satisfying the $*$-condition and every $n_0>0$ , there exists a $n>n_0$ 
for which 
$$\frac{a_n}{a_{n+1}}< 1+\frac{r}{n\ln(n)}+o\left(\frac{1}{n\ln(n)}\right).$$
\end{corollary}
\begin{proof}
It suffices to use the contrapositive of Lemma \ref{LEMA}.
\end{proof}

The main result of this note is as follows. 
\begin{theorem}\label{main}
If there exists a pair $(\{b_n\},r)$ satisfying the $*$-condition  for some $r\geq 2$ then
$$\liminf_{n\to\infty} g_n\leq r.$$
\end{theorem}
\begin{proof}
Suppose that there exists a pair $(\{b_n\},r)$ satisfying the $*$-condition with $r\geq 2$.
From Lemma \ref{primes} and Corollary \ref{coro} we have that 
$$\frac{p_{n+1}}{p_{n}}< 1+\frac{r}{n\ln(n)}+o\left(\frac{1}{n\ln(n)}\right),$$
for infinitely many values of $n$.  That is,
$$g_n< p_n\left(\frac{r}{n\ln(n)}+o\left(\frac{1}{n\ln(n)}\right)\right),$$
for infinitely many values of $n$. 

By Lemma \ref{PNT}, $n\sim n\ln(n)$, hence
$$g_n< n\ln(n)\left(\frac{r}{n\ln(n)}+o\left(\frac{1}{n\ln(n)}\right)\right),$$
for infinitely many values of $n$. Therefore
$$\liminf_{n\to\infty} g_n\leq r.$$

\end{proof}

\section{Final remarks}

In this note we have presented an alternative and simple  method to deal with the Polignac's conjecture.  In particular,  Theorem \ref{main} indicates that
if one find a pair $(\{b_n\},r)$ satisfying the $*$-condition with $r= 2$ then  the twin-prime conjecture would be true. More generally  if, for each $r\geq 2$ even, one finds a pair   $(\{b_n\},r)$ satisfying the $*$-condition
then  the Polignac's conjecture would be  true.

\end{document}